\documentclass[10pt,reqno]{amsart}

\usepackage{fixltx2e}                       
\usepackage[usenames,dvipsnames]{xcolor}
\usepackage{fancyhdr}
\usepackage{amsmath,amsfonts,amsbsy,amsgen,amscd,mathrsfs,amssymb,amsthm}
\usepackage{mathtools}
\usepackage[font=small,margin=0.25in,labelfont={sc},labelsep={colon}]{caption}
\usepackage{tikz}
\usepackage{enumerate}
\usepackage{graphicx}

\usepackage{bm} 

\usepackage[colorlinks=true]{hyperref}

\newtheorem{thm}{Theorem}[section]
\newtheorem{lem}[thm]{Lemma}

\newtheorem{prop}[thm]{Proposition}
\newtheorem{cor}[thm]{Corollary}

\theoremstyle{definition}

\newtheorem{defn}[thm]{Definition}

\newtheorem{rem}[thm]{Remark}

\numberwithin{equation}{section} 
\numberwithin{figure}{section}
\numberwithin{table}{section}

\newcommand{\absconv}{\mathop{\mathrm{absconv}}}

\begin{document}

\title{Chaining, Interpolation, and Convexity}
\author{Ramon van Handel}
\address{Sherrerd Hall Room 227, Princeton University, Princeton, NJ 
08544, USA}
\thanks{Supported in part by NSF grant
CAREER-DMS-1148711 and by the ARO through PECASE award 
W911NF-14-1-0094.}
\email{rvan@princeton.edu}

\begin{abstract}
We show that classical chaining bounds on the suprema of random processes 
in terms of entropy numbers can be systematically improved when the 
underlying set is convex: the entropy numbers need not be computed for the 
entire set, but only for certain ``thin'' subsets.  This phenomenon arises 
from the observation that real interpolation can be used as a natural 
chaining mechanism. Unlike the general form of Talagrand's generic 
chaining method, which is sharp but often difficult to use, the resulting 
bounds involve only entropy numbers but are nonetheless sharp in many 
situations in which classical entropy bounds are suboptimal.  Such bounds 
are readily amenable to explicit computations in specific examples, and we 
discover some old and new geometric principles for the control of chaining 
functionals as special cases. \end{abstract}



\maketitle

\thispagestyle{empty}

\section{Introduction}

A remarkable achievement of modern probability theory is the development 
of sharp connections between the boundedness of random processes and the 
geometry of the underlying index set.  Perhaps the most fundamental 
result in this direction is the characterization of boundedness of 
Gaussian processes due to Talagrand.

\begin{thm}[\cite{Tal14}]
\label{thm:mm}
Let $(X_t)_{t\in T}$ be a centered Gaussian process and denote by
$d(t,s) = (\mathbf{E}|X_t-X_s|^2)^{1/2}$ the associated natural metric on 
$T$.  Then
$$
	\mathbf{E}\bigg[\sup_{t\in T}X_t\bigg]
	\asymp
	\gamma_2(T) :=
	\inf\sup_{t\in T}\sum_{n\ge 0}2^{n/2}d(t,T_n),
$$
where the infimum is taken over all sequences of sets $T_n$
with cardinality $|T_n|<2^{2^n}$.
\end{thm}

The quantity $\gamma_2(T)$ captures precisely what aspect of the geometry 
of the metric space $(T,d)$ controls the suprema of Gaussian processes: it 
quantifies the degree to which $T$ can be approximated by a sequence of 
increasingly fine nets $T_n$. While we quote this particular result for 
concreteness, the structure that is expressed by Theorem \ref{thm:mm}, 
called the generic chaining, extends far beyond the theory of Gaussian 
processes and has a substantial impact on various problems in probability, 
functional analysis, statistics, and theoretical computer science.  An 
extensive development of this theory and its implications can be found in 
\cite{Tal14}.

Theorem \ref{thm:mm} provides a powerful general principle for the study 
of the suprema of random processes.  However, when presented with any 
specific situation, it often proves to be remarkably difficult to control 
$\gamma_2(T)$ efficiently.  Theorem \ref{thm:mm} can only give sharp 
results if one is able to construct a nearly optimal sequence of nets 
$T_n$, a task that is significantly complicated by the multiscale nature 
of $\gamma_2(T)$.  The aim of this paper is to exhibit some surprisingly 
elementary principles that make it possible to obtain sharp control of 
$\gamma_2(T)$ in various interesting examples, and that shed new light on 
the underlying geometric phenomena.

There are essentially two general approaches that have been used to 
control $\gamma_2(T)$.  The simplest and by far the most useful approach 
is obtained by bringing the supremum over $t\in T$ inside the sum in the 
definition of $\gamma_2(T)$.  This yields
$$
	\gamma_2(T) \le \sum_{n\ge 0} 2^{n/2} e_n(T),
$$
where the entropy number $e_n(T)$ is defined as the smallest 
$\varepsilon>0$ such that there is an $\varepsilon$-net in $T$ of 
cardinality less than $2^{2^n}$.  This bound, due to Dudley \cite{Dud67}, 
long predates Theorem \ref{thm:mm} and has found widespread use.  Its 
utility stems from the fact that controlling entropy numbers only requires 
us to approximate the set $T$ at a single scale, for which numerous 
methods are available; see, e.g., \cite{KT59,ET96,AGM15}. Unfortunately, 
Dudley's bound can fail to be sharp even in the simplest examples, such as 
ellipsoids in Hilbert space.  In fact, the supremum of a random process on 
$T$ cannot in general be understood in terms of the entropy numbers of 
$T$: one can easily construct two such sets with comparable entropy 
numbers on which a Gaussian process behaves very differently \cite{Sud69}. 
It is therefore a crucial feature of Theorem \ref{thm:mm} that the use of 
entropy numbers is replaced by a genuinely multiscale form of 
approximation.  The construction of such a multiscale approximation in 
any given situation is however a highly nontrivial task.

The main approach that has been developed for the latter purpose is 
Talagrand's growth functional machinery \cite{Tal14} that forms the core 
of the proof of Theorem \ref{thm:mm}.  To show that $\gamma_2(T)$ is upper 
bounded by the expected supremum of the Gaussian process, the proof of 
Theorem \ref{thm:mm} constructs nets $T_n$ by means of a greedy 
partitioning scheme that uses the Gaussian process itself $G(A) := 
\mathbf{E}[\sup_{t\in A}X_t]$ as an objective function. It turns out that 
the success of this proof relies on the properties of Gaussian processes 
only through the validity of a single ``growth condition'' of the 
functional $G$.  If one can design another functional $F$ that mimics this 
property of Gaussian processes, then the same proof also yields an upper 
bound on $\gamma_2(T)$ in terms of $F(T)$.  An important example of such a 
construction is the proof that $\gamma_2(T)$ is strictly smaller than 
Dudley's bound when $T$ is a $q$-convex body \cite[\S 4.1]{Tal14}.  It is 
generally far from obvious, however, how a functional $F$ can be designed, 
and successful application of this approach requires considerable 
ingenuity.

In this paper, we develop a new approach that is intermediate between 
these two extremes.  The central insight of this paper is that it is 
possible to improve systematically on Dudley's bound without giving up the 
formulation in terms of entropy numbers.  Of course, as was noted above, 
we cannot expect to improve on Dudley's bound in a general setting in 
terms of the entropy numbers of $T$ itself.  Instead, we will show that 
when $T$ 
is a convex set, the entropy numbers $e_n(T)$ in Dudley's bound can be 
replaced by the entropy numbers of certain ``thin'' subsets that can be 
substantially smaller than $T$. (The convexity assumption is not essential 
for our approach, but leads to a cleaner statement of the results.)

To illustrate this idea, let us begin by stating a useful form of such a 
result. Let $(X,\|\cdot\|)$ be a Banach space, and let $B\subset X$ be a 
symmetric compact convex set.  We denote by $\|\cdot\|_B$ the gauge of 
$B$, and by $\|\cdot\|_B^*$ and $\|\cdot\|^*$ the dual norms on $X^*$. 
In this setting, we will always choose the distance $d$ in the 
definitions of $\gamma_2(B)$ and $e_n(B)$ to be the one generated by the 
norm $d(x,y):=\|x-y\|$.

\begin{thm}
\label{thm:main}
Let $B\subset (X,\|\cdot\|)$ be a symmetric compact convex set, and
define
$$
	B_t := \{ y\in B : \exists\,z\in X^*\mbox{ such that }
	\langle z,y\rangle = \|y\|_B,~
	\|z\|_B^*\le 1,~\|z\|^*\le t\}.
$$
Then we have for any $a>0$
$$
	\gamma_2(B) \lesssim \frac{1}{a} + 
	\sum_{n\ge 0} 2^{n/2} e_n(B_{a2^{n/2}}).
$$
\end{thm}

The bound of Theorem \ref{thm:main} proves to be sharp
in many situations in which Dudley's bound is suboptimal, and often 
provides a simple explanation for why this is the case.  At the same 
time, Theorem \ref{thm:mm} is typically no more difficult to apply than 
Dudley's bound, as the ``thin'' subsets $B_t\subseteq B$ that appear in 
this bound can be found in quite explicit form.  For example, if $B$ is a 
smooth symmetric convex body in $\mathbb{R}^d$, then it is a classical 
fact that $\nabla\|x\|_B$ is the unique norming functional for the norm 
$\|\cdot\|_B$ at the point $x$, so that we can simply write
$$
	B_t = \{y\in B: \|\nabla\|y\|_B\|^* \le t\}.
$$
Such expressions are readily amenable to explicit computations.

One of the nice features of Theorem \ref{thm:main} is that the phenomenon 
that it describes arises in a completely elementary fashion.  To 
understand its origin, let us sketch the simple idea behind the 
proof.  The basic challenge in controlling $\gamma_2(B)$ is to approximate 
the unit ball of the norm $\|\cdot\|_B$ in terms of another norm 
$\|\cdot\|$.  It proves to be useful to connect these two norms using an 
idea that is inspired by real interpolation of Banach spaces \cite{BS88}.   
To this end, define Peetre's $K$-functional
$$
	K(t,x) := \inf_y\{\|y\|_B + t\|x-y\|\} =
	\|\pi_t(x)\|_B + t\|x-\pi_t(x)\|,
$$
where $\pi_t(x)$ is any minimizer in the definition of $K(t,x)$ 
(assume for simplicity that we work in a finite-dimensional Banach space 
to avoid trivial technicalities).  It is easily seen that 
$\lim_{t\to\infty}K(t,x)=\|x\|_B$, $K(0,x)=0$, and
$\frac{d}{dt}K(t,x)=\|x-\pi_t(x)\|$ (the latter follows by observing that
$\|x-\pi_t(x)\|$ is a supergradient of the concave
function $t\mapsto K(t,x)$, so it must equal $\frac{d}{dt}K(t,x)$ a.e.; 
see Proposition \ref{prop:superdudley} below.)
We therefore obtain by the fundamental theorem of calculus
$$
	\|x\|_B =
	\int_0^\infty \|x-\pi_t(x)\|\,dt \asymp
	\sum_{n\ge 0} 2^{n/2}\|x-\pi_{2^{n/2}}(x)\|,
$$
where the last step follows from a Riemann sum approximation of 
the integral.  This leads immediately to the following observation:
if we define the sets
$$
	B_t := \{\pi_t(x):x\in B\},
$$
then we have shown that
$$
	\sup_{x\in B}\sum_{n\ge 0}2^{n/2}d(x,B_{2^{n/2}}) \lesssim 1.
$$
In other words, we see that a natural chaining mechanism is in fact built 
into the real interpolation method: we automatically generate a multiscale 
approximation of $B$ in terms of the sets $B_t$.  In order to 
bound $\gamma_2(B)$, it remains to choose a finite net with the 
appropriate cardinality inside each of the sets $B_t$.  (While it may not 
be immediately obvious, the definition of $B_t$ given in Theorem 
\ref{thm:main} is none other than the dual formulation of the 
definition of $B_t$ as a set of minimizers.)

It should be clear at this point that convexity is not essential in the 
construction using real interpolation: convexity only enters the proof of 
Theorem \ref{thm:main} in order to obtain the convenient formulation of 
the sets $B_t$.  In section \ref{sec:superdudley}, we first prove a 
general form of Theorem \ref{thm:main} that is applicable in any metric 
space; we also formulate the results for more general 
$\gamma_p$-functionals that appear when the generic chaining method is 
applied to non-Gaussian processes.  We then specialize to the convex 
setting and derive the dual formulation of $B_t$. In section \ref{sec:ex}, 
we illustrate the power of Theorem \ref{thm:main} in a number of explicit 
examples.  We also illustrate by means of an example that Theorem 
\ref{thm:main} does not always give sharp results.

Theorem \ref{thm:main} improves on Dudley's bound by replacing the entropy 
numbers of $B$ by the entropy numbers of the smaller sets $B_t$.  A rather 
different improvement arises when $B$ is $q$-convex, for which 
Talagrand shows that \cite[\S 4.1]{Tal14}
$$
        \gamma_2(B) \lesssim
        \Bigg[
        \sum_{n\ge 0} \big(2^{n/2} e_n(B)\big)^{q/(q-1)}
        \Bigg]^{(q-1)/q}.
$$
This bound involves only the entropy numbers of the set $B$ itself, and 
appears at first sight to be quite different in nature than Theorem 
\ref{thm:main}. Nonetheless, we show in section \ref{sec:geom} that this 
fundamental result is a direct consequence of Theorem 
\ref{thm:main}.  Roughly speaking, we will see that the $q$-convexity 
assumption forces the sets $B_t$ to be much smaller than the original set 
$B$ in the sense that $e_n(B_t) \lessapprox t^{1/(q-1)}e_n(B)^{q/(q-1)}$. 
In fact, it turns out there is nothing particularly special about  
uniform convexity: Talagrand's result is a special case of a 
more general geometric phenomenon that will be developed in section 
\ref{sec:geom}.  As another illustration of this phenomenon, we will show 
that Talagrand's bound for $q$-convex bodies holds verbatim for 
$\ell_q$-balls in Banach spaces with an unconditional basis for every 
$1<q<\infty$.  Note that such sets are only $2$-convex rather than 
$q$-convex when $1<q<2$, so that the behavior of $\ell_q$-balls
is evidently not explained by uniform convexity.

The connection between interpolation and generic chaining appears in 
hindsight to be entirely natural.  Many generic chaining constructions 
(that appear in \cite{Tal14,Tal96}, for example) have a flavor of 
interpolation, and even the multiscale notion of approximation that is 
intrinsic to the definition of $\gamma_2(T)$ has appeared 
independently in interpolation theory in the study of approximation spaces 
\cite{Pie81,DP88,PS72}.  To the best of the author's knowledge, however, 
the results of this paper are the first to explicitly develop this 
connection. It would be interesting to understand whether broader
interactions exist between these areas of probability and analysis.

\section{Chaining, Interpolation, and Convexity}
\label{sec:superdudley}

The aim of this section is to develop the basic connections between 
chaining, interpolation, and convexity that lie at heart of this paper. In 
section \ref{sec:metric}, we develop an abstract chaining principle that 
holds in any metric space.  In section \ref{sec:convex}, we specialize to 
the convex setting and complete the proof of Theorem \ref{thm:main}.

\subsection{Chaining and interpolation}
\label{sec:metric}

In this section, let $(X,d)$ be any metric space. We begin by defining 
formally the notions of entropy numbers and Talagrand's 
$\gamma_p$-functionals. The case $p=2$ arises in the context of Gaussian 
processes together with the associated natural metric, cf.\ Theorem 
\ref{thm:mm}; however, other values of $p$ and more general metrics can 
arise for other random processes \cite{Tal14}.

\begin{defn}
\label{defn:talg}
For any $A\subseteq X$ and $n\ge 0$, define the entropy number
$$
	e_n(A) := \inf_{|\tilde A|<2^{2^n}}\sup_{x\in A}
	d(x,\tilde A),
$$
and define for $p>0$ the $\gamma_p$-functional
$$
	\gamma_p(A) := 
	\inf_{|\tilde A_n|<2^{2^n}}\sup_{x\in A}\sum_{n\ge 0}
	2^{n/p}d(x,\tilde A_n).
$$
(The approximating sets $\tilde A_n\subseteq X$ are not necessarily 
subsets of $A$.)
\end{defn}

Fix a set $A\subseteq X$ for the remainder of this section.
To measure the size of $A$, we introduce a penalty 
function $f:X\to\mathbb{R}_+\cup\{+\infty\}$ that may in principle be 
chosen arbitrarily.  Consider the corresponding optimization problem
$$
	K(t,x) := \inf_{y\in X}\{f(y) + td(x,y)\}
$$
for every $t\ge 0$ and $x\in A$.  We will assume for simplicity that the 
infimum in this optimization problem is attained for every $t\ge 0$ and 
$x\in A$, and denote by $\pi_t(x)$ any choice of minimizer in the 
definition of $K(t,x)$.  (It is a trivial exercise to extend our results 
to the setting where $\pi_t(x)$ is a near-minimizer, but such an extension 
will not be needed in the sequel.)  We now define for every $t\ge 0$ the 
set
$$
	A_t := \{\pi_t(x):x\in A\}.
$$

\begin{rem}
In the present formulation, $A_t$ is not necessarily a subset of $A$.  
However, it is natural to choose a penalty function $f$ such that
$A=\{x:f(x)\le 1\}$, in which case evidently $A_t\subseteq A$ (because
$f(\pi_t(x)) \le K(t,x) \le f(x)$).
\end{rem}

The following result lies at the heart of this paper.  In the sequel, we 
write $a\lesssim b$ if $a\le Cb$ for a universal constant $C$, and 
$a\asymp b$ if $a\lesssim b$ and $b\lesssim a$.  We indicate explicitly 
when the universal constant depends on some parameter in the problem.

\begin{prop}
\label{prop:superdudley}
In the setting of this section, we have for every $a>0$
$$
	\gamma_p(A) \lesssim
	\frac{1}{a}\sup_{x\in A}f(x) +
	\sum_{n\ge 0} 2^{n/p} e_n(A_{a2^{n/p}}),
$$
where the universal constant depends on $p$ only.
\end{prop}

\begin{proof}
We can assume without loss of generality that $f$ is uniformly bounded 
on $A$.  Thus $0\le K(t,x)\le f(x)<\infty$ for every $x\in A$ and $t\ge 
0$.  Moreover, $t\mapsto K(t,x)$ is clearly a concave function for every 
$x\in A$.  We now use some basic facts about univariate concave functions 
\cite[Chapter I]{HUL93}.  First, we note that
\begin{align*}
	K(t,x)-K(s,x)
	&=
	\inf_{y\in X}\{f(y) + td(x,y)\} -
	f(\pi_s(x)) - sd(x,\pi_s(x)) \\
	&\le (t-s)d(x,\pi_s(x))
\end{align*}
for all $t,s\ge 0$, so that $d(x,\pi_s(x))$ is a supergradient of 
$t\mapsto K(t,x)$ at $t=s$.  As a bounded concave function is absolutely 
continuous, we obtain
$$
	K(T,x) = K(0,x) + \int_0^T d(x,\pi_t(x))\,dt
$$
for every $T\ge 0$ and $x\in A$.  In particular, we can estimate
$$
	\int_0^\infty d(x,\pi_t(x))\,dt \le f(x)
$$
for every $x\in A$.  We also recall that the derivative of a concave 
function is nonincreasing, so that we can discretize the integral as 
follows:
\begin{align*}
	f(x) &\ge \int_0^a d(x,\pi_t(x))\,dt +
	\sum_{n\ge 1}\int^{a2^{n/p}}_{a2^{(n-1)/p}}
	d(x,\pi_t(x))\,dt \\
	&\ge
	(1 - 2^{-1/p})\,a\sum_{n\ge 0} 2^{n/p}d(x,\pi_{a2^{n/p}}(x)),
\end{align*}
where we used that $t\mapsto d(x,\pi_t(x))$ is
nonincreasing in the last step.

It remains to discretize the minimizers $\pi_t(x)$.
By the definition of entropy numbers, we can choose for every $n\ge 0$ a 
set $\tilde A_n\subseteq X$ such that $|\tilde A_n|<2^{2^n}$ and
$$
	\sup_{x\in A_{a2^{n/p}}}d(x,\tilde A_n) 
	\le 2e_n(A_{a2^{n/p}}).
$$
We can therefore estimate
\begin{align*}
	\gamma_p(A) &\le
	\sup_{x\in A}\sum_{n\ge 0}2^{n/p}
	d(x,\tilde A_n)
	\\
	&\le
	\sup_{x\in A}\sum_{n\ge 0}2^{n/p}d(x,\pi_{a2^{n/p}}(x))
	+
	\sum_{n\ge 0}2^{n/p}
	\sup_{x\in A}
	d(\pi_{a2^{n/p}}(x),\tilde A_n)
	\\
	&\lesssim
	\frac{1}{a}\sup_{x\in A}f(x) +
	\sum_{n\ge 0}2^{n/p} e_n(A_{a2^{n/p}}),
\end{align*}
which completes the proof.
\end{proof}

\begin{rem}
Suppose we replace the penalty $f$ by an equivalent penalty $\tilde 
f\asymp f$. Then the first term in the bound of Proposition 
\ref{prop:superdudley} only changes by a universal constant, but the 
second term might change substantially as the definition of the sets $A_t$ 
is highly nonlinear.  This highlights the nontrivial nature of the choice 
of penalty.  Similarly, the bound of Theorem \ref{thm:main} could 
potentially give better results if we replace $B$ by an equivalent set 
$c\tilde B\subseteq B\subseteq C\tilde B$.  Note that the same 
phenomenon arises when applying the growth functional machinery of 
\cite{Tal14}: the growth condition is not preserved if we choose 
an equivalent functional. This appears to be an inherent difficulty that 
arises in the control of chaining functionals.  \end{rem}

\subsection{Convexity}
\label{sec:convex}

While Proposition \ref{prop:superdudley} provides a very general chaining 
principle in metric spaces, it is not immediately obvious how to apply 
this result in any given situation.  The problem is that the sets $A_t$ 
that appear in the previous section are defined implicitly as families of 
solutions to certain optimization problems; in the absence of a more 
explicit characterization, the computation of the entropy numbers 
$e_n(A_{a2^{n/p}})$ can be a challenging problem.  To address this 
problem, we specialize our results from this point onwards to the case 
where the set of interest is convex and where the penalty function is 
chosen to be the associated gauge. The convexity assumption makes it 
possible to obtain a dual formulation of the sets of optimizers that is 
readily amenable to explicit computations.  The advantages of this 
formulation will be amply illustrated in the following sections.

We now introduce the setting that will be used throughout the remainder of 
this paper.  Let $(X,\|\cdot\|)$ be a Banach space, and let $B\subset X$ 
be a symmetric compact convex set. The metric $d$ that appears in the 
definitions of the entropy numbers $e_n(B)$ and the functionals 
$\gamma_p(B)$ (cf.\ Definition \ref{defn:talg}) will always be chosen to 
be defined by the norm $d(x,y):=\|x-y\|$ on the underlying Banach space. 
The gauge (Minkowski functional) of $B$ will be denoted $\|\cdot\|_B$, that is,
$$
	\|x\|_B := \inf\{s\ge 0:x\in sB\}
$$
for $x\in X$.  Denote by $\|\cdot\|_B^*$ and $\|\cdot\|^*$ the 
associated dual gauge and norm, that is,
$$
	\|z\|^*_B := \sup_{\|x\|_B\le 1}\langle z,x\rangle
	= \sup_{x\in B}\langle z,x\rangle,\qquad\quad
	\|z\|^* := \sup_{\|x\|\le 1}\langle z,x\rangle
$$
for $z\in X^*$.  The key point of this section is the following duality 
result, which shows that the minimizers of the $K$-functional in the 
convex setting define a form of projection onto an explicitly defined 
scale of subsets $B_t\subseteq B$.

\begin{prop}
\label{prop:dual}
For every $t\ge 0$, there is a map $\pi_t:B\to B$ such that:
\begin{enumerate}[(i)]
\item $\pi_t(x)$ is a minimizer for Peetre's $K$-functional for every
$x\in B$:
$$
	K(t,x) := \inf_{y\in X}\{\|y\|_B+t\|x-y\|\} =
	\|\pi_t(x)\|_B + t\|x-\pi_t(x)\|.
$$
\item The set of minimizers 
$$
	B_t := \{\pi_t(x):x\in B\}
$$
can be characterized as
$$
	B_t = \{ y\in B : \exists\,z\in X^*\mbox{ such that }
	\langle z,y\rangle = \|y\|_B,~
	\|z\|_B^*\le 1,~\|z\|^*\le t\}.
$$
\item We have $\pi_t(x)=x$ for every $x\in B_t$.
\end{enumerate}
\end{prop}

\begin{proof}
The result holds trivially for $t=0$, so we fix $t>0$ in the sequel.

\textbf{Step 1.}
Let $B_K := \mathrm{conv}(B\cup\frac{1}{t}B_\sim)$, where $B_\sim$ is the 
closed unit ball in $(X,\|\cdot\|)$.  For completeness, we recall the 
proof of the elementary fact that $K(t,x) = \|x\|_{B_K}$ for every $x\in 
X$, where $\|\cdot\|_{B_K}$ denotes the gauge of $B_K$.

Suppose first that $K(t,x)<r$, so there exists $y\in X$ with 
$\|y\|_B+t\|x-y\|<r$.  Then writing $x=\lambda x_1+ \mu x_2$ with 
$x_1=y/\|y\|_B$ and $x_2=(x-y)/t\|x-y\|$ readily implies that 
$\|x\|_{B_K}<r$.  In the converse direction, suppose that $\|x\|_{B_K}<r$, 
so that $x=\lambda x_1+\mu x_2$ for some $|\lambda|+|\mu|<r$, $x_1\in B$, 
$x_2\in \frac{1}{t}B_\sim$.  Then choosing $y=\lambda x_1$ in the 
definition of $K(t,x)$ shows that $K(t,x)<r$.

\textbf{Step 2.}
We now establish the existence of a minimizer in the definition of
$K(t,x)$ for every $x\in X$.  This is a direct consequence of the previous 
step and the compactness of $B$.  Indeed, as $B$ is compact, the set $B_K$ 
is closed.  Thus $K(t,x)=r$ implies $x\in rB_K$, so there exist 
$|\lambda|+|\mu|\le r$ and $x_1\in B$, $x_2\in \frac{1}{t}B_\sim$ such 
that $x=\lambda x_1+\mu x_2$.  It follows that $y=\lambda x_1$ is a 
minimizer for $K(t,x)$, as
$$
	K(t,x) \le
	\|\lambda x_1\|_B + t\|\mu x_2\| \le r = K(t,x).
$$

\textbf{Step 3.}
Define the set
$$
	B_t' := \{y\in B: K(t,y)=\|y\|_B\}.
$$
We can characterize this set by duality.  Indeed, note
that
$$
	K(t,y) = \sup\{\langle z,y\rangle:z\in X^*,~
			\|z\|_B^*\le 1,~\|z\|^*\le t\},
$$
where we have used the polar identity $B_K^\circ = B^\circ \cap 
tB_\sim^\circ$.  Moreover, the supremum is 
attained at some point $z\in X^*$ by the Hahn-Banach theorem.
Therefore, if $y\in B_t'$, then
there exists $z\in X^*$ such that $\langle z,y\rangle=\|y\|_B$,
$\|z\|_B^*\le 1$, and $\|z\|^*\le t$.  Conversely, if $y\in B$ is such 
that a point $z$ satisfying the latter properties exists, then
$\|y\|_B = \langle z,y\rangle\le K(t,y)\le \|y\|_B$
so that $y\in B_t'$. Thus we have
$$
	B_t' = \{ y\in B : \exists\,z\in X^*\mbox{ such that }
	\langle z,y\rangle = \|y\|_B,~
	\|z\|_B^*\le 1,~\|z\|^*\le t\}.
$$

\textbf{Step 4.}
Define the map $\pi_t:B\to B$ as follows.  For $x\in B_t'$, we set
$\pi_t(x)=x$.  For $x\not\in B_t'$, we choose $\pi_t(x)$ to be any 
minimizer in the definition of $K(t,x)$.  We are going to verify that each 
of the claims in the statement of the Proposition hold.

Let us first note that $\pi_t$ does indeed map $B$ into itself.  For $x\in 
B_t'$, this is true by construction.  For $x\not\in B_t'$, this is true 
because $\|\pi_t(x)\|_B \le K(t,x) \le \|x\|_B$.  Moreover, note that when 
$x\in B_t'$, by construction $y=x=\pi_t(x)$ is a minimizer in the 
definition of $K(t,x)$.  We have therefore established part (i).

To prove parts (ii) and (iii), it suffices to show that $B_t=B_t'$.  That
$B_t'\subseteq B_t$ is obvious from the fact that $\pi_t(x)=x$ for
$x\in B_t'\subseteq B$.  To establish the converse inclusion, we argue as 
follows.
Fix $x\in B$, and choose $z\in X^*$ such that $K(t,x)=\langle z,x\rangle$,
$\|z\|_B^*\le 1$, and $\|z\|^*\le t$.  By the bipolar theorem, we can 
write
$$
	\langle z,\pi_t(x)\rangle \le 
	\|\pi_t(x)\|_B =
	\langle z,\pi_t(x)\rangle +
	\langle z,x-\pi_t(x)\rangle - t\|x-\pi_t(x)\|  \le
	\langle z,\pi_t(x)\rangle.
$$
This implies that $\pi_t(x)\in B_t'$, and thus $B_t\subseteq B_t'$.
\end{proof}

\begin{rem}
When $B$ is a symmetric convex body in a finite-dimensional Banach space, 
the details of the proof of Proposition \ref{prop:dual} simplify 
significantly.  It is an instructive exercise to give a quick proof 
in this case using subdifferential calculus. 
\end{rem}

The proof of Theorem \ref{thm:main} in the introduction now follows 
trivially.  For future reference, we formulate the analogous result for 
$\gamma_p$-functionals.

\begin{cor}
\label{cor:main}
Let $B\subset (X,\|\cdot\|)$ be a symmetric compact convex set, and
define
$$
	B_t := \{ y\in B : \exists\,z\in X^*\mbox{ such that }
	\langle z,y\rangle = \|y\|_B,~
	\|z\|_B^*\le 1,~\|z\|^*\le t\}.
$$
Then we have for any $a>0$
$$
	\gamma_p(B) \lesssim \frac{1}{a} + 
	\sum_{n\ge 0} 2^{n/p} e_n(B_{a2^{n/2}}),
$$
where the universal constant depends on $p$ only.
\end{cor}

\begin{proof}
This is simply the combined statement of Proposition 
\ref{prop:superdudley}, where we choose the penalty $f(x)=\|x\|_B$ and
distance $d(x,y)=\|x-y\|$, and Proposition \ref{prop:dual}.
\end{proof}

We end this section by emphasizing a remark that was also made in the 
introduction.  Recall that a symmetric convex set $B\subset X$ is called 
smooth if for every $x\in X$, $x\ne 0$ there is a unique $z\in X^*$ so 
that $\langle z,x\rangle=\|x\|_B$ and $\|z\|_B^*\le 1$, cf.\ \cite{Bea82}.

\begin{cor}
\label{cor:grad}
Let $B$ be a symmetric convex body in a finite-dimensional Banach space
$(X,\|\cdot\|)$, and denote by $\partial\|y\|_B$ the subdifferential of 
$\|y\|_B$.  Then
$$
	B_t = \Big\{ y\in B : \inf_{z\in\partial\|y\|_B}\|z\|^*\le t
	\Big\}.
$$
In particular, if $B$ is smooth, then
$$
	B_t = \{ y\in B : \|\nabla\|y\|_B\|^*\le t\}.	
$$
\end{cor}

\begin{proof}
It is a classical fact that
$\partial\|y\|_B = \{z\in X^*:\langle z,y\rangle=\|y\|_B,~\|z\|_B^*\le 1\}$,
so that the result follows readily from Proposition \ref{prop:dual}; cf.\ 
\cite[Chapter VI]{HUL93}.
\end{proof}

The explicit nature of Corollary \ref{cor:grad} is particularly useful
in computations.

\section{Examples}
\label{sec:ex}

The aim of this section is to illustrate the utility of Theorem 
\ref{thm:main} in explicit computations by investigating some simple 
but conceptually interesting examples.  As our goal is to develop insight 
into the phenomenon described by Theorem \ref{thm:main}, we have avoided 
unnecessary distractions by restricting attention to situations in which 
existing entropy estimates can be used.

We write $\|x\|_r:=[\sum_i|x_i|^r]^{1/r}$, and denote by $e_1,\ldots,e_d$ 
the standard basis in $\mathbb{R}^d$. Throughout this section, we work in 
Euclidean space $(\mathbb{R}^d,\|\cdot\|)$ where $\|\cdot\|:=\|\cdot\|_2$. 
The concrete choice of the Euclidean norm is not important for our theory, 
but is made in order to enable explicit computations and is natural in the 
setting of Gaussian processes (as it corresponds to the canonical choice 
$X_t=\langle t,g\rangle$ in Theorem \ref{thm:mm}, where $g$ is a standard 
Gaussian vector in $\mathbb{R}^d$). Some of the examples developed here 
will be revisited in section \ref{sec:geom} in a much more general 
setting.

\subsection{\texorpdfstring{$\ell_q$}{lq}-Ellipsoids}
\label{sec:lqellips}

The classical example of a situation where Dudley's bound fails to be 
sharp is that of ellipsoids in Hilbert space.  In this section, we will 
investigate the following more general situation.
Given scalars $1<q<\infty$ and $b_1\ge b_2\ge\cdots\ge b_d>0$, let 
$B\subset\mathbb{R}^d$ be the $\ell_q$-ellipsoid whose gauge is given by
$$
        \|x\|_B = \Bigg[\sum_{i=1}^d 
	\bigg(\frac{|x_i|}{b_i}\bigg)^q\Bigg]^{1/q}.
$$
We will show that Theorem \ref{thm:main} yields the following optimal
bound.

\begin{prop}
\label{prop:lqellips}
In the setting of this section, we have   
$$
        \gamma_2(B) \lesssim \Bigg(\sum_{i=1}^d b_i^{q/(q-1)}\Bigg)^{(q-1)/q},
$$
where the universal constant depends on $q$ only.
\end{prop}

Of course, this result can easily be obtained from Theorem \ref{thm:mm}, 
but our aim is to provide a geometric proof that explains why the result 
is true.

In order to apply either Dudley's bound or Theorem \ref{thm:main}, we 
will require suitable estimates on the entropy numbers of $\ell_q$-ellipsoids.  
The behavior of these entropy numbers is investigated in detail in a 
classic paper by Carl \cite{Car81} (in the special case of 
$\ell_2$-ellipsoids a much more elementary approach can be found in 
\cite[\S 2.5]{Tal14}).  For future reference, we record a more general 
form of the main result of Carl than is presently needed.  The following 
can be read off from the proof of \cite[Theorem 2]{Car81}. (While the 
result of Carl is formulated only for $r\ge 1$, the proof extends directly 
to the case $0<r<1$ if we replace \cite[Theorem 1]{Car81} by 
\cite[Proposition 3.2.2]{ET96}.)

\begin{lem}[\cite{Car81}]
\label{lem:carl}
Given $0<r<\infty$, $1/s > (1/2-1/r)_+$, $0<u<\infty$, and 
scalars $c_1\ge c_2\ge \cdots\ge c_d>0$, the $\ell_r$-ellipsoid
$C=\{x\in\mathbb{R}^d:\|(x_i/c_i)\|_r\le 1\}$ satisfies
$$
	\sum_{n\ge 0} \big(2^{n(1/s+1/r-1/2)}e_n(C)\big)^u
	\asymp
	\sum_{k=1}^d (k^{1/s-1/u}c_k)^u
$$
where the universal constant depends on $r,s,u$ only.
\end{lem}

Applying this result with $r=q$, $1/s = 1-1/q$, and $u=1$ yields
$$
	\sum_{n\ge 0} 2^{n/2} e_n(B) \asymp
	\sum_{k=1}^d k^{-1/q}b_k.
$$
We therefore see immediately that Dudley's bound is suboptimal for 
$\ell_q$-ellipsoids: Dudley's bound is much larger than $\gamma_2(B)$, 
say, when $b_k=k^{-(q-1)/q}(\log k)^{-1}$.

To obtain a sharp bound, we will apply Theorem \ref{thm:main}.
The crux of the matter is to control the sets $B_t$.  In the present 
setting, this is exceedingly simple and gives a vivid illustration of 
where the improvement over Dudley's bound comes from.

\begin{proof}[Proof of Proposition \ref{prop:lqellips}]
Note that $B$ is a smooth convex body with
$$
	\frac{\partial\|y\|_B}{\partial y_k} = 
	\frac{1}{b_k^q}
	\frac{|y_k|^{q-1}}{\|y\|_B^{q-1}}\mathop{\mathrm{sign}}(y_k).
$$
Thus Corollary \ref{cor:grad} gives
$$
	B_t = \{y\in B:\|y\|_{C}\le t^{1/(q-1)}\|y\|_B\} \subseteq
	t^{1/(q-1)}C,
$$
where
$$
        \|y\|_C = \Bigg[\sum_{i=1}^d 
	\bigg(\frac{|y_i|}{b_i^{q/(q-1)}}\bigg)^{2q-2}\Bigg]^{1/(2q-2)}.
$$
Substituting $B_t\subseteq t^{1/(q-1)}C$ into Theorem \ref{thm:main} and
optimizing over $a>0$ yields
$$
	\gamma_2(B) \lesssim
	\Bigg(
	\sum_{n\ge 0}2^{nq/(2q-2)} e_n(C)
	\Bigg)^{(q-1)/q}.
$$
The conclusion follows by applying Lemma \ref{lem:carl} with
$r=2q-2$ and $s=u=1$.
\end{proof}

The key point of the proof of Proposition \ref{prop:lqellips} is that each 
subset $B_t$ of the $\ell_q$-ellipsoid $B$ is contained in a dilation of 
the much ``thinner'' $\ell_{2q-2}$-ellipsoid $C$: the lengths of the 
semiaxes of $C$ have been raised to the power $q/(q-1)$ as compared to 
those of $B$.  This is precisely why we obtain the correct powers of $b_i$ 
inside the sum in Proposition \ref{prop:lqellips}.  The author sees no 
obvious way to explain this miracle other than that it drops out of the 
trivial explicit computation performed above.  However, a deeper 
understanding of the geometry of the sets $B_t$ for $\ell_q$-ellipsoids 
will be obtained in a much more general setting in section \ref{sec:geom}.

\begin{rem}
There exist two previous geometric proofs of Proposition 
\ref{prop:lqellips} for special values of $q$.  The first, in \cite[\S 
15.6]{LT91}, gives a delicate manual construction of an equivalent 
formulation of $\gamma_2(B)$ for $q=2$.  The second, in \cite[\S 
4.1]{Tal14}, deduces the result for $2\le q<\infty$ from a more general 
bound for uniformly convex bodies that is proved using the growth 
functional machinery. We will revisit the latter idea in section 
\ref{sec:geom}, where we will also see that uniform convexity fails to 
explain the behavior of $\ell_q$-ellipsoids for $1<q<2$.  That we have 
obtained a sharp bound for every value of $q$ with the same proof 
therefore hides the fact that $\ell_q$-ellipsoids can have a very 
different geometry for different values of $q$.  
\end{rem}

\begin{rem}
The universal constant in Proposition \ref{prop:lqellips} must necessarily 
depend on $q$: if this were not the case, then we would obtain 
$\gamma_2(B)\lesssim b_1$ in the limit $q\downarrow 1$ which is easily 
seen to be false by Theorem \ref{thm:mm}.  Unfortunately, the entropy 
estimates provided by Lemma \ref{lem:carl} are not sufficiently accurate 
to recover the correct behavior as $q\downarrow 1$.  This is not a 
deficiency of Theorem \ref{thm:main}, however: the case $q=1$ is of 
particular interest in its own right and will be investigated in 
the next section. \end{rem}

\subsection{Octahedra}
\label{sec:oct}

In this section, we investigate the limiting case $q=1$ of the example 
developed in the previous section.  That is, given scalars $b_1\ge b_2\ge 
\cdots \ge b_d>0$, we investigate the octahedron $B\subset\mathbb{R}^d$
defined by
$$
	B = \absconv\{b_ie_i:i=1,\ldots,d\}.
$$
It is not difficult to show that Dudley's bound is suboptimal in this 
setting \cite[Exercise 2.2.15]{Tal14}.
We will show that Theorem \ref{thm:main} yields the following optimal 
bound.

\begin{prop}
\label{prop:oct}
In the setting of this section, we have
$$
	\gamma_2(B) \lesssim 
	\Sigma := \max_{i\le d} b_i\sqrt{\log(i+1)}.
$$
\end{prop}

Of course, this result could easily be obtained from Theorem \ref{thm:mm}, 
and a rather difficult geometric proof using growth functionals can be 
found in \cite[\S 8]{Tal96}.  However, the point for our purposes is that 
this result follows in a completely elementary fashion from Theorem 
\ref{thm:main}.  To apply the latter, let us first identify the sets 
$B_t$.

\begin{lem}
\label{lem:octspars}
For any $t\ge 0$, we have
$$
	B_t =
	\Bigg\{
	y\in B : \sum_{i=1}^d \frac{\mathbf{1}_{y_i\ne 0}}{b_i^2} \le 
	t^2
	\Bigg\}.
$$
\end{lem}

\begin{proof}
While $\|\cdot\|_B$ is not smooth, we can easily compute its 
subdifferential:
$$
	\partial\|y\|_B =
	\{z\in\mathbb{R}^d:z_i=\mathop{\mathrm{sign}}(y_i)/b_i
	\mbox{ if }y_i\ne 0,~|z_i|\le 1/b_i\mbox{ if }y_i=0\}.
$$
We therefore obtain
$$
	\inf_{z\in\partial\|y\|_B}\|z\|^2 =
	\sum_{i=1}^d \frac{\mathbf{1}_{y_i\ne 0}}{b_i^2},
$$
and the result follows from Corollary \ref{cor:grad}.
\end{proof}

Lemma \ref{lem:octspars} shows that the sets $B_t$ are very thin indeed: 
they consist of sparse vectors.  Controlling the entropy numbers of 
such sets is an easy exercise; for each fixed sparsity pattern we can 
discretize using a standard volumetric argument, while counting the number 
of sparsity patterns is a matter of simple combinatorics.

\begin{lem}
\label{lem:octent}
There is a universal constant $c>0$ such that for all $n\ge 0$
$$
	e_n(B_{c2^{n/2}/\Sigma}) \lesssim 2^{-n}b_1.
$$
\end{lem}

\begin{proof}
Fix $n\ge 0$.
As $1/b_i^2 \ge \log(i+1)/\Sigma^2$ by definition, we have
$$
	B_t \subseteq C_t :=
	\Bigg\{
	y\in B : \sum_{i=1}^d \log(i+1)\,\mathbf{1}_{y_i\ne 0} \le 
	\Sigma^2t^2
	\Bigg\}.
$$
It suffices to control the entropy numbers of the larger set 
$C_{c2^{n/2}/\Sigma}$.

Let us begin with some counting.  Denote by $\mathcal{I}$ the family
of all admissible sparsity patterns of $y\in C_{c2^{n/2}/\Sigma}$, that 
is, $\mathcal{I}$ is the family of all $I\subseteq[d]$ such that
$$
	\sum_{i\in I}\log(i+1) \le c^22^{n}.
$$
Denote by $\mathcal{I}_k\subseteq\mathcal{I}$ the family of all 
$I\in\mathcal{I}$ with cardinality $|I|=k$.  Let us bound the number 
of such sets.  Setting $c  := \sqrt{\log 2}/2$, we can estimate
$$
	|\mathcal{I}_k| 
	=
	\sum_{|I|=k}
	\mathbf{1}_{I\in\mathcal{I}}
	=
	\sum_{|I|=k}
	\mathbf{1}_{
	\prod_{i\in I}(i+1)^2\le 2^{2^{n-1}}}
	\le
	2^{2^{n-1}} 
	\sum_{|I|=k}
	\prod_{i\in I}\frac{1}{(i+1)^2}.
$$
The right-hand side can be bounded as follows:
$$
        \sum_{|I|=k}
        \prod_{i\in I}\frac{1}{(i+1)^2} =
	\sum_{1\le\ell_1<\ell_2<\cdots<\ell_k\le d}\,
	\prod_{i=1}^k\frac{1}{(\ell_i+1)^2}
	\le
	\prod_{i=1}^k\sum_{\ell\ge i}\frac{1}{(\ell+1)^2}
	<\frac{1}{k!},
$$
where we have used that
$$
	\sum_{\ell\ge i}\frac{1}{(\ell+1)^2} <
	\sum_{\ell\ge i}\int_\ell^{\ell+1}\frac{1}{x^2}\,dx
	= \int_i^\infty \frac{1}{x^2}\,dx
	= \frac{1}{i}.
$$
We have therefore shown that $|\mathcal{I}_k|<2^{2^{n-1}}/k!$.

Let $\varepsilon\le b_1$ be a constant to be chosen later on.
For every $I\in\mathcal{I}$, 
choose a minimal $\varepsilon$-net $T_I$ for the Euclidean ball in 
$\mathbb{R}^I$ with radius $b_1$, and denote by $T$ the 
union of all these sets $T_I$.  Evidently $T$ is a $\varepsilon$-net for 
$C_{c2^{n/2}/\Sigma}$.  Let us estimate its cardinality.  A standard 
volumetric argument yields \cite[Corollary 4.1.15]{AGM15}
$$
	|T_I| \le \bigg(
	\frac{3b_1}{\varepsilon}
	\bigg)^{|I|}.
$$
We can therefore estimate
$$
	|T| \le
	\sum_{k=0}^d \bigg(
        \frac{3b_1}{\varepsilon}
        \bigg)^{k} |\mathcal{I}_k| <
	2^{2^{n-1}} e^{3b_1/\varepsilon}.
$$
If we choose $\varepsilon = (6/\log 2)\,2^{-n}b_1$, we find that
$|T|<2^{2^n}$ which establishes the claim whenever $2^n\ge 6/\log 2$
(as we assumed that $\varepsilon\le b_1$ in the volumetric estimate).  For 
$2^n<6/\log 2$, simply note the trivial bound 
$e_n(C_{c2^{n/2}/\Sigma})\le\mathrm{diam}(B)\le 2b_1$.
\end{proof}

With this entropy estimate in hand, the proof of Proposition 
\ref{prop:oct} is an immediate consequence of Lemma \ref{lem:octent} and 
Theorem \ref{thm:main} with $a=c/\Sigma$.

\subsection{A counterexample}

The aim of this section is to show that Theorem \ref{thm:main} does 
not always give sharp results.  As the example that we will discuss is a 
conceptually important one, let us briefly consider this example in a 
broader context.

A remarkable consequence of Theorem \ref{thm:mm} is that 
$\gamma_2(\mathop{\mathrm{conv}}(T)) \asymp \gamma_2(T)$ for any 
(non-convex) subset $T\subseteq\mathbb{R}^d$ of Euclidean space:
as the supremum of a linear function over a convex set is attained at
an extreme point, Theorem \ref{thm:mm} yields
$$
	\gamma_2(T)\asymp 
	\mathbf{E}\bigg[
	\sup_{x\in T}\langle x,g\rangle
	\bigg]=
	\mathbf{E}\bigg[
	\sup_{x\in \mathrm{conv}(T)}\langle x,g\rangle
	\bigg]\asymp \gamma_2(\mathop{\mathrm{conv}}(T))
$$
(here $g$ denotes a standard Gaussian vector in $\mathbb{R}^d$). It is a 
long-standing open problem to understand the geometric mechanism behind 
this fundamental fact; cf.\ \cite[\S 2.4]{Tal14}.  By using a known
device \cite[Theorem 2.4.18]{Tal14}, one can reduce this problem to the 
following special case: it suffices to give a geometric proof of the fact 
that for any $x_1,\ldots,x_n\in\mathbb{R}^d$ such that $\|x_1\|\ge 
\|x_2\|\ge\cdots\ge\|x_n\|>0$, we have
$$
	\gamma_2(B) \lesssim \max_{i\le n}\|x_i\|\sqrt{\log(i+1)},\qquad
	\quad
	B = \absconv\{x_i:i=1,\ldots,n\}.
$$
We solved this problem in the previous section under the additional 
assumption that the vectors $x_i$ are orthogonal.  It is not known, 
however, how this conclusion can be established in the absence of the 
orthogonality assumption.  The results of this paper originated in an
attempt by the author to understand this issue.  We will presently 
illustrate that Theorem \ref{thm:main} does not directly resolve this 
problem.

The example that we will consider is defined as follows.  Fix 
$0<\varepsilon<1$ and let $u=d^{-1/2}\mathbf{1}$, where $\mathbf{1}$ is 
the vector of ones (note that $\|u\|=1$). We consider the set
$$
	B = \absconv\{x_i:i=1,\ldots,d\},\qquad\quad
	x_i = e_i + \varepsilon u.
$$
This is a small perturbation of the example in the previous section where 
all vertices of the simplex have been shifted along the diagonal.  One can 
show as in \cite[Exercise 2.2.15]{Tal14} that $\gamma_2(B)\asymp\sqrt{\log 
d}$, while Dudley's bound is of order $(\log d)^{3/2}$.

We claim that Theorem \ref{thm:main} does not improve on Dudley's bound in 
the present setting: the sets $B_t$ are not sufficiently small to gain any 
improvement.  This unfortunate conclusion is contained in the following 
lemma.

\begin{lem}
\label{lem:oops}
We have $B_t\supseteq \mathop{\mathrm{conv}}\{x_i:i=1,\ldots,d\}$
for all $t\ge 1/\varepsilon$.
\end{lem}

\begin{proof}
Let $V=\sum_{i=1}^d x_i\otimes e_i$ be the square matrix whose columns 
are the vectors $x_i$.  Note that $V$ is invertible, and we have $\|x\|_B 
= \|V^{-1}x\|_1$.  Therefore
$$
	\partial\|x\|_B =
	(V^*)^{-1}\partial\|V^{-1}x\|_1 \ni
	(V^*)^{-1}\mathop{\mathrm{sign}}(V^{-1}x),
$$
where $\mathop{\mathrm{sign}}(z)$ operates entrywise on a vector $z$ and
we set $\mathop{\mathrm{sign}}(0):=1$.  In particular,
$$
	B_t \supseteq \{x\in B:
	\mathop{\mathrm{sign}}(V^{-1}x)\in tV^*B_\sim\}
$$
by Corollary \ref{cor:grad}, where $B_\sim$ denotes the Euclidean unit 
ball in $\mathbb{R}^d$.

Now note that if $x\in \mathop{\mathrm{conv}}\{x_i:i=1,\ldots,d\}$, then
$V^{-1}x$ has nonnegative entries and thus 
$\mathop{\mathrm{sign}}(V^{-1}x)=\mathbf{1}$.  It therefore suffices to
show that $\mathbf{1}\in tV^*B_\sim$ whenever
$t\ge 1/\varepsilon$.  But this is a simple consequence of the definition 
of $V$, as
$$
	tV^*v = \mathbf{1}\quad\mbox{for}\quad
	v = \frac{u}{t(\varepsilon+d^{-1/2})}
$$
and clearly $\|v\|\le 1$ when $t\ge 1/\varepsilon$.  This completes the 
proof.
\end{proof}

Let $\Delta^{d-1}$ be the standard simplex in $\mathbb{R}^d$.  Lemma 
\ref{lem:oops} shows that $B_t\supseteq \Delta^{d-1}+\varepsilon u$
whenever $t\ge 1/\varepsilon$.  Setting 
$n_{a,\varepsilon}=(2\log_2(1/a\varepsilon))_+$,
we can estimate
$$
	\sum_{n\ge 0} 2^{n/2} e_n(B_{a2^{n/2}}) \ge
	\sum_{n\ge n_{a,\varepsilon}}
	2^{n/2} e_n(\Delta^{d-1}) \gtrsim
	(\log d)^{3/2} -
	Cn_{a,\varepsilon}\sqrt{\log d}
$$
for some constant $C>0$, where we used that $e_n(\Delta^{d-1})\gtrsim 
2^{-n/2}\sqrt{\log d}$ for $n\lesssim\log d$ \cite[Exercise 2.2.15]{Tal14}.
We have therefore shown that Theorem \ref{thm:main} does not improve on 
Dudley's bound in this example unless $\varepsilon$ is polynomially small 
in $d$.

\begin{rem}
Of course, the example described in this section is sufficiently simple 
that we can make some manual adjustments to obtain a sharp geometric 
construction.  Indeed, we clearly have $B\subset B_1+B_2$ where $B_1$ 
denotes the $\ell_1$-ball in $\mathbb{R}^d$ and $B_2=\{\alpha 
u:|\alpha|\le\varepsilon\}$ is one-dimensional. Theorem 
\ref{thm:main} gives a sharp generic chaining construction for $B_1$, 
while a trivial discretization of $\alpha$ suffices to control $B_2$.  We 
can then glue together the generic chaining constructions for $B_1$ and 
$B_2$ by summing the corresponding nets.  It is not clear, however, 
how one could construct such a decomposition in the general setting 
described at the beginning of this section.
\end{rem}

\section{Geometry and Entropy Contraction}
\label{sec:geom}

In the previous section, we illustrated the utility of Theorem 
\ref{thm:main} in specific examples.  The computations hinge, however, on 
a sufficiently explicit description of the sets $B_t$, which may not 
always be available in more general situations. For example, if we 
consider the examples of the previous section under general norms, it may 
be nontrivial to control the sets $B_t$ directly.  It is therefore of 
interest to develop more systematic methods to control the geometry of the 
sets $B_t$.

As a prototype of what one might hope for, let us reconsider the setting 
of $\ell_q$-ellipsoids in Hilbert space.  Theorem \ref{thm:main} bounds 
$\gamma_2(B)$ in terms of the entropy numbers of the sets $B_t$, which we 
computed explicitly in section \ref{sec:lqellips}.  However, Lemma 
\ref{lem:carl} suggests that the correct behavior of $\gamma_2(B)$ in this 
example can also be expressed in terms of the entropy numbers of $B$ 
itself: we easily verify that
$$
	\gamma_2(B) \asymp \Bigg[
	\sum_{n\ge 0}\big(2^{n/2}e_n(B)\big)^{q/(q-1)}
	\Bigg]^{(q-1)/q}.
$$
The appearance of such a bound is not a coincidence. Talagrand has shown 
that an upper bound of this form holds for any $q$-convex set $B$ \cite[\S 
4.1]{Tal14}: as $\ell_q$-ellipsoids are $\max(2,q)$-convex, this provides 
an alternative explanation for the behavior of $\ell_q$-ellipsoids in the 
case $2\le q<\infty$. One of the insights to be developed in this section 
is that this fundamental property of $q$-convex sets is fully explained by 
Theorem \ref{thm:main}.  Roughly speaking, we will show that the 
$q$-convexity assumption forces the sets $B_t$ to be much smaller than $B$ 
itself
in the sense that $e_n(B_t)\lessapprox t^{1/(q-1)}e_n(B)^{q/(q-1)}$, from 
which the above bound is easily deduced.  More generally, this phenomenon 
suggests that the chaining principle for general convex sets given by
Theorem \ref{thm:main} can be significantly simplified in the 
presence of additional geometric structure.

It turns out that there is nothing special about $q$-convexity per se, but 
that the entropy contraction phenomenon illustrated above arises from a 
much more general geometric mechanism.  We develop a general formulation 
of this idea in section \ref{sec:geompr}.  We then demonstrate how the 
requisite structure arises in two distinct settings: the case of 
$q$-convex sets is developed in section \ref{sec:uniconv}, while the case 
of $\ell_q$-balls in Banach spaces with an unconditional basis is 
developed in section \ref{sec:lquncond}.

\subsection{A geometric principle}
\label{sec:geompr}

Let $(X,\|\cdot\|)$ be a Banach space and let $B\subset X$ be a symmetric 
compact convex set.  The sets $B_t$ are defined as in Theorem 
\ref{thm:main}.  The following geometric principle is the main result of 
this section.

\begin{thm}
\label{thm:geompr}
Let $q>1$ and $K>0$ be given constants, and suppose that
$$
        \|y-z\|_B^q  \le Kt\|y-z\|\quad\mbox{for every }y,z\in B_t,~
        t\ge 0.
$$
Then
$$
	\gamma_p(B) \lesssim
	\Bigg[
	\sum_{n\ge 0} \big(2^{n/p} e_n(B)\big)^{q/(q-1)}
	\Bigg]^{(q-1)/q},
$$
where the universal constant depends on $p$, $q$, and $K$ only.
\end{thm}

Like Theorem \ref{thm:main}, the message of Theorem \ref{thm:geompr} is 
that the behavior of $\gamma_p(B)$ is strictly better than would be 
expected from Dudley's bound.  Unlike Theorem \ref{thm:main}, however, the 
presence of additional geometric structure allows us to bound 
$\gamma_p(B)$ only in terms of the entropy numbers of $B$ itself. This 
bound could therefore be applied even without an explicit description of 
$B_t$.  Of course, there is no free lunch: the assumption of Theorem 
\ref{thm:geompr} requires us to understand the metric structure of the 
sets $B_t$.  Fortunately, we will see in the sequel that there are 
interesting situations in which this can be accomplished without 
explicitly computing the sets $B_t$.

\begin{rem}
\label{rem:geomintu}
Before we turn to the proof of Theorem \ref{thm:geompr}, it is instructive 
to consider the significance of the geometric assumption of 
Theorem \ref{thm:geompr}.  Observe that we always have, regardless of any 
assumptions, the following simple fact:
$$
	\|y\|_B \le t\|y\|\quad\mbox{for every }y\in B_t,~t\ge 0.
$$
Indeed, if $z\in X^*$ is as in the definition of $B_t$, then
$$
	\|y\|_B = \langle z,y\rangle \le \|z\|^*\|y\| \le
	t\|y\|.
$$
We therefore see that by construction, an element $y\in B_t$ with small 
norm must be contained in a small dilation $y\in t\|y\|B$ of the 
original set $B$.  The assumption of Theorem \ref{thm:geompr} asks that a 
weaker form of this property hold not only for norms, but also for 
distances: that is, if $y,z\in B_t$, then $y-z\in (Kt\|y-z\|)^{1/q}B$.  
This does not follow automatically from the corresponding property for 
norms, as it is typically not true that $B_t-B_t\subseteq cB_{ct}$ for 
some constant $c$.  Nonetheless, this intuition proves to be useful as it 
will help us identify how the requisite geometric structure arises.
\end{rem}

The main idea behind the proof of Theorem \ref{thm:geompr} is the
following observation.

\begin{lem}
\label{lem:genprent}
Suppose that the assumption of Theorem \ref{thm:geompr} holds.  Then
$$
	e_{n+1}(B_t) \le
	(K t\, e_n(B_t))^{1/q}e_n(B)
	\quad\mbox{for every }n\ge 0,~t\ge 0.
$$
\end{lem}

\begin{proof}
Fix $\varepsilon>0$. By the definition 
of entropy numbers, we can cover $B_t$ by less than $2^{2^n}$ balls of 
radius $(1+\varepsilon)e_n(B_t)$.  By our assumption, each of these balls 
(intersected with $B_t$) is contained in a translate of $sB$ with $s\le 
(1+\varepsilon)^{1/q}(Kt\,e_n(B_t))^{1/q}$.  Therefore, each of these 
balls can be further covered by less than $2^{2^n}$ balls of radius 
$(1+\varepsilon)s\, e_n(B)$.  We have now covered $B_t$ by less than 
$2^{2^n}\cdot 2^{2^n} = 2^{2^{n+1}}$ balls of radius $\le 
(1+\varepsilon)^{1+1/q}(Kt\,e_n(B_t))^{1/q}e_n(B)$.  Letting 
$\varepsilon\downarrow 0$ completes the proof.
\end{proof}

An annoying feature of Lemma \ref{lem:genprent} is that the entropy number 
on the left-hand side is $e_{n+1}(B_t)$ rather than $e_n(B_t)$.  If it 
were the case that $e_n(B_t)\lesssim e_{n+1}(B_t)$ (that is, if we knew 
\emph{a priori} that the entropy numbers do not decay too quickly), then 
we could simplify the conclusion of Lemma \ref{lem:genprent} to
$$
	e_{n}(B_t) \lesssim
	t^{1/(q-1)}e_n(B)^{q/(q-1)}.
$$
This expression quantifies in the present setting in what sense the sets 
$B_t$ are much smaller than the original set $B$.
From this expression, it would be easy to conclude the result of Theorem 
\ref{thm:geompr}: substituting the above bound into Theorem 
\ref{thm:main} yields
$$
	\gamma_p(B) \lesssim
	\frac{1}{a} + a^{1/(q-1)}
	\sum_{n\ge 0} \big(2^{n/p}e_n(B)\big)^{q/(q-1)},
$$
and the conclusion of Theorem \ref{thm:geompr} would follow
by optimizing over $a>0$.
The main technical issue in the proof of Theorem \ref{thm:geompr} is to 
show that its conclusion remains valid even when the regularity assumption 
$e_n(B_t)\lesssim e_{n+1}(B_t)$ does not hold, which we do by means of a 
routine dyadic regularization argument.

\begin{proof}[Proof of Theorem \ref{thm:geompr}]
Fix a constant $\lambda>0$ to be chosen at a later stage.  For any set 
$C$, we introduce the regularized entropy numbers $d_n(C)\ge e_n(C)$ as
$$
	d_n(C) := \max_{0\le k\le n} 2^{\lambda(k-n)}e_k(C).
$$
Using Lemma \ref{lem:genprent}, we estimate
\begin{align*}
	d_n(B_t) &\le \max_{0\le k\le n+1} 
	2^{\lambda(k-n)}e_k(B_t) \\
	&\le
	2^{-\lambda n}\mathrm{diam}(B) +
	2^{\lambda} \max_{0\le k\le n}
        2^{\lambda(k-n)}e_{k+1}(B_t) \\
	&\lesssim
	2^{-\lambda n}\mathrm{diam}(B) +
	2^{\lambda}t^{1/q}\max_{0\le k\le n}
        2^{\lambda(k-n)}e_k(B_t)^{1/q}e_k(B) \\
	&\le
	2^{-\lambda n}\mathrm{diam}(B) +
	2^{\lambda}t^{1/q}d_n(B_t)^{1/q}
	\max_{0\le k\le n}
        2^{\lambda(k-n)(q-1)/q}e_k(B).
\end{align*}
Therefore, using $a^{1/q}b^{(q-1)/q} \le a/q + b(q-1)/q$, we obtain
$$
	d_n(B_t) \lesssim
	2^{-\lambda n}\mathrm{diam}(B)
	+
	2^{\lambda q/(q-1)}
	t^{1/(q-1)}
        \max_{0\le k\le n}
        2^{\lambda(k-n)}e_k(B)^{q/(q-1)}.
$$
In particular, we can crudely bound
\begin{multline*}
	\sum_{n\ge 0} 2^{n/p} e_n(B_{a2^{n/p}}) 
	\lesssim
	\mathrm{diam}(B)
	\sum_{n\ge 0} 2^{n/p}2^{-\lambda n} + \mbox{} \\ 
	a^{1/(q-1)}
	2^{\lambda q/(q-1)}
	\sum_{n\ge 0} 
	2^{nq/(q-1)p} 2^{-\lambda n}
        \sum_{0\le k\le n}
        2^{\lambda k}e_k(B)^{q/(q-1)}.
\end{multline*}
In order for the sums to converge we must choose $\lambda > q/(q-1)p$, so 
we fix for concreteness $\lambda=2q/(q-1)p$ (the precise value of 
$\lambda$ does not matter).  This yields
\begin{align*}
	&\sum_{n\ge 0} 2^{n/p} e_n(B_{a2^{n/p}}) \\
	&\lesssim
	\mathrm{diam}(B) +
	a^{1/(q-1)}
	\sum_{n\ge 0}
        2^{-nq/(q-1)p} 
        \sum_{0\le k\le n}
        \big(2^{2k/p}e_k(B)\big)^{q/(q-1)} \\
	& =
        \mathrm{diam}(B) +
        a^{1/(q-1)}  
        \sum_{k\ge 0} 
        \sum_{n\ge k} 
        2^{-nq/(q-1)p}
        \big(2^{2k/p}e_k(B)\big)^{q/(q-1)} \\
	& \lesssim
        \mathrm{diam}(B) +
        a^{1/(q-1)}
	\sum_{k\ge 0}
        \big(2^{k/p}e_k(B)\big)^{q/(q-1)}.
\end{align*}
Applying Corollary \ref{cor:main} and optimizing over $a>0$ yields
$$
	\gamma_p(B) \lesssim
	\mathrm{diam}(B) + 
	\Bigg[
	\sum_{n\ge 0} \big(2^{n/p} e_n(B)\big)^{q/(q-1)}
	\Bigg]^{(q-1)/q}.
$$
It remains to note that $\mathrm{diam}(B)\le 2e_0(B)$, so that the first 
term can be absorbed in the second at the expense of the universal 
constant.
\end{proof}

\begin{rem}
An inspection of the proof shows that the universal constant in Theorem 
\ref{thm:geompr} blows up as $q\downarrow 1$.  It would be interesting to 
understand whether there is an analogue of Theorem \ref{thm:geompr} that 
holds in the limiting case $q=1$: that is, whether there is a general 
geometric mechanism that ensures the sharp bound
$$
	\gamma_p(B) \asymp \sup_{n\ge 0}2^{n/p}e_n(B)
$$
(that the right-hand side is a lower bound on $\gamma_p(B)$ is trivial).
This situation is illustrated by the example of section \ref{sec:oct}: in 
this case both the assumption and the conclusion of Theorem 
\ref{thm:geompr} hold for $q=1$ (the assumption holds by Remark 
\ref{rem:geomintu} and $B_t-B_t\subseteq 2B_{\sqrt{2}t}$, while the 
conclusion can be deduced from \cite[Exercise 2.2.15]{Tal14}), but 
Theorem \ref{thm:geompr} is not sufficiently sharp to capture 
this example. \end{rem}

\subsection{Uniformly convex sets}
\label{sec:uniconv}

In this section, we exhibit an important situation where the assumption of 
Theorem \ref{thm:geompr} can be verified by imposing additional geometric 
structure on the set $B$: we show that the assumption holds when $B$ is 
$q$-convex.  This recovers a fundamental result of Talagrand
\cite[\S 4.1]{Tal14}.

Let $(X,\|\cdot\|)$ be any Banach space, and let $B\subset X$ be a 
symmetric convex set.  As usual, we denote by $\|\cdot\|_B$ the gauge of 
$B$.  We recall the following definition.
        
\begin{defn}
Let $q\ge 2$. A symmetric convex set $B$ is called \emph{$q$-convex} if
$$
        \bigg\|\frac{x+y}{2}\bigg\|_B \le 1-\eta\|x-y\|_B^q
$$
for all $x,y\in B$, where $\eta>0$ is an absolute constant.
\end{defn}

We will prove the following result.
  
\begin{cor}[\cite{Tal14}]
\label{cor:qconvex}
Let $B$ be a symmetric convex set in a Banach space $(X,\|\cdot\|)$, and
assume that $B$ is $q$-convex (with constant $\eta$).  Then
$$
        \gamma_p(B) \lesssim
        \Bigg[
        \sum_{n\ge 0} \big(2^{n/p} e_n(B)\big)^{q/(q-1)}
        \Bigg]^{(q-1)/q},
$$
where the universal constant depends on $p$, $q$, and $\eta$ only.
\end{cor}

To connect this result to the explicit computations in section 
\ref{sec:lqellips}, we recall that $\ell_q$-ellipsoids are 
$\max(2,q)$-convex \cite{Bea82}.  This shows that the case $2\le q<\infty$ 
of Proposition \ref{prop:lqellips} is in fact a manifestation of the much 
more general phenomenon described by Corollary \ref{cor:qconvex}: we 
emphasize that the present result requires no assumption of any kind on 
the norm $\|\cdot\|$.  On the other hand, it is impossible for a convex 
set to be $q$-convex with $q<2$ (Hilbert space is maximally convex),
so that uniform convexity cannot explain the behavior of 
$\ell_q$-ellipsoids for $q<2$.  We will nonetheless see in the next 
section that the latter case can also be understood as a manifestation of
the general geometric principle described by Theorem \ref{thm:geompr}.

We prove Corollary \ref{cor:qconvex} by verifying the assumption of 
Theorem \ref{thm:geompr}.

\begin{lem}
\label{lem:qconvbt}
Let $B$ be a $q$-convex set and $t\ge 0$.  Then
$$
	\|y-z\|_B^q  \lesssim t\|y-z\|\quad\mbox{for every }y,z\in B_t,
$$
where the universal constant depends on $q$ and $\eta$ only.
\end{lem}

We will give two different proofs of this lemma.  The first proof is 
pedestrian, but perhaps not very intuitive.  The second proof is more 
intuitive, as it is close in spirit to the intuition developed in Remark 
\ref{rem:geomintu}; however, this proof requires us to use an alternative 
(but equivalent) formulation of the $q$-convexity property.

\begin{proof}[First proof]
By Proposition \ref{prop:dual}, we have
$\pi_t(y)=y$ for $y\in B_t$.  Thus
$$
	\|y\|_B = 
	\inf_u \{\|u\|_B + t\|y-u\|\}
	\le
	\bigg\|\frac{y+z}{2}\bigg\|_B +
	t \bigg\|\frac{y-z}{2}\bigg\|
$$
for any $y,z\in B_t$.
Similarly, exchanging the role of $y$ and $z$, we obtain
$$
	1 \le
	\bigg\|\frac{y+z}{2\gamma}\bigg\|_B +
	t\bigg\|\frac{y-z}{2\gamma}\bigg\|,
	\qquad
	\gamma:=\|y\|_B\vee \|z\|_B.
$$
But note that $\|y/\gamma\|_B\le 1$ and
$\|z/\gamma\|_B\le 1$ by the definition of $\gamma$.
Therefore, applying the $q$-convexity assumption to the first term
on the right yields
$$
	\|y-z\|_B^q  \le
	\frac{\gamma^{q-1}}{2\eta}\,t\|y-z\|
$$
for any $y,z\in B_t$.  The proof is completed by noting that $\gamma\le 1$.
\end{proof}

\begin{proof}[Second proof]
An equivalent characterization of the $q$-convexity property is as 
follows \cite[Corollary 1]{Xu91}: $B$ is $q$-convex if and only if
$$
	\langle j_y-j_z,y-z\rangle
	\gtrsim \|y-z\|_B^q
$$
for all $j_y\in J_y:=\{u\in X^*: \langle u,y\rangle = 
\|y\|_B^q,~\|u\|_B^*\le\|y\|_B^{q-1}\}$ and $j_z\in J_z$, where
the universal constant depends on $q,\eta$ only.  Note that $J_y$ is none 
other than the subdifferential of the map $y\mapsto\|y\|_B^q/q$ (cf.\ 
Corollary \ref{cor:grad}), so this characterization is rather intuitive: 
$B$ is $q$-convex precisely when the map $y\mapsto\|y\|_B^q$ exhibits a 
uniform improvement over the usual first-order condition for convexity.

With this formulation in hand, the lemma follows easily.  Let $y,z\in 
B_t$.  By definition of $B_t$, we can choose $u_y\in X^*$ with
$\langle u_y,y\rangle =\|y\|_B$, $\|u_y\|^*_B\le 1$, $\|u_y\|^*\le t$.
Choose $u_z\in X^*$ analogously.  Setting
$j_y = u_y\|y\|_B^{q-1}$ and $j_z=u_z\|z\|^{q-1}_B$ gives
$$
	\|y-z\|^q_B \lesssim \langle j_y-j_z,y-z\rangle
	\le \|j_y-j_z\|^*\|y-z\| \le
	2t\|y-z\|.
$$
This completes the proof.
\end{proof}

It is now trivial to complete the proof of Corollary \ref{cor:qconvex}.

\begin{proof}[Proof of Corollary \ref{cor:qconvex}]
We may as well assume that $B$ is compact: if $B$ is not precompact, the 
right-hand side of the desired inequality is infinite and there is nothing 
to prove; if $B$ is precompact, there is no loss of generality in assuming 
that it is also closed.  It remains to apply Theorem \ref{thm:geompr} 
and Lemma \ref{lem:qconvbt}.
\end{proof}

\subsection{\texorpdfstring{$\ell_q$}{lq}-balls and unconditional bases}
\label{sec:lquncond}

We have seen in the previous section that uniform convexity cannot explain 
the behavior of $\ell_q$-ellipsoids in Hilbert space that was observed in 
section \ref{sec:lqellips}.  We will presently show that this behavior is 
nonetheless a manifestation of the general geometric principle of Theorem 
\ref{thm:geompr}.  It will follow immediately that the same behavior 
persists in a much larger family of Banach spaces (but not in a setting as 
general as for $q$-convex sets).

To understand what is going on, let us take inspiration from the second 
proof of Lemma \ref{lem:qconvbt} (and from Remark \ref{rem:geomintu}).
For any $x\in X$, choose any point $j_x\in X^*$ be such that
$\langle j_x,x\rangle=\|x\|_B^q$ and $\|j_x\|_B^*\le\|x\|_B^{q-1}$.
As
$$
	\|y-z\|_B^q = \langle j_{y-z},y-z\rangle \le
	\|j_{y-z}\|^*\|y-z\|,
$$
the assumption of Theorem \ref{thm:geompr} would follow if we 
could show that $\|j_{y-z}\|^*\lesssim t$ whenever $y,z\in B_t$.  
We can always choose $\|j_x\|^*\le t$ when $x\in B_t$, but this 
does not in itself yield the desired result: $y,z\in B_t$ does not imply 
$y-z\in B_t$.  

To obtain the desired bound, we must find a relation between $j_{y-z}$ 
and $j_y,j_z$.  The $q$-convexity assumption provides the inequality 
$\langle j_{y-z},y-z\rangle \lesssim \langle j_y-j_z,y-z\rangle$, 
which is particularly convenient for this purpose.  However, this is by 
no means the only way to achieve our goal.  In the case of 
$\ell_q$-ellipsoids, we will use a completely different geometric 
property: in this case we observe that $|j_{y-z}| \lesssim 
|j_y|+|j_z|$ coordinatewise.  This simple device allows us to reach the 
same conclusion as in the $q$-convex case as long as the dual norm 
$\|\cdot\|^*$ respects the coordinatewise ordering.

We proceed to make this idea precise.  We first recall the class of 
Banach spaces that possess the desired monotonicity properties \cite[\S 
3.1]{AK06}.

\begin{defn}
\label{def:uncond}
Let $(X,\|\cdot\|)$ be a Banach space and let $\{e_n\}$ be a basis for
$X$.  The basis is said to be unconditional with constant $K$ if
$$
	\Bigg\|\sum_{n=1}^N a_ne_n\Bigg\| \le K
	\Bigg\|\sum_{n=1}^N b_ne_n\Bigg\|
$$
for all $N\in\mathbb{N}$ and scalars $a_n,b_n\in\mathbb{R}$
such that $|a_n|\le|b_n|$ for all $n$.  
\end{defn}

We recall for future reference that if $\{e_n\}$ is an unconditional basis 
in $X$ with constant $K$, then the biorthogonal sequence $\{e_n^*\}$ is an 
unconditional basic sequence in $X^*$ with the same constant $K$
\cite[Proposition 3.2.1]{AK06}.

\begin{rem}
The notion of a $K$-unconditional basis is often defined in a slightly
different way than we have done above: a basis is unconditional with 
constant $K$ if
$$
	\Bigg\|\sum_{n=1}^N \varepsilon_nb_ne_n\Bigg\| \le K
        \Bigg\|\sum_{n=1}^N b_ne_n\Bigg\|
$$
for all $N\in\mathbb{N}$, $b_n\in\mathbb{R}$, and 
$\varepsilon_n\in\{-1,+1\}$, that is, if the norm of $\sum_{n=1}^N b_ne_n$
is approximately invariant to sign changes of the coefficients $b_n$.
The more general property of Definition \ref{def:uncond} is however 
readily deduced from this alternative definition (for example, by
choosing random signs $\varepsilon_n$ such that 
$a_n=\mathbf{E}[\varepsilon_nb_n]$).
\end{rem}

In the following let $(X,\|\cdot\|)$ be a Banach space and let $\{e_n\}$ 
be an unconditional basis with constant $K$.  Fix $1<q<\infty$, and 
define the $\ell_q$-ball $B\subset X$ as follows:
$$
	B = \Bigg\{
	\sum_{i=1}^d z_i e_i : 
	\sum_{i=1}^d |z_i|^q \le 1
	\Bigg\}
$$
(our result will be independent of $d$, and therefore extends readily to 
infinite dimension).  Note that the $\ell_q$-ellipsoids 
considered in section \ref{sec:lqellips} correspond to the 
special case where $\{e_i\}$ is the standard basis in $\mathbb{R}^d$ and
$\|x\|^2 = \sum_i b_i^2x_i^2$.

\begin{cor}
\label{cor:lqelluncond}
In the setting of this section, we have
$$
        \gamma_p(B) \lesssim
        \Bigg[
        \sum_{n\ge 0} \big(2^{n/p} e_n(B)\big)^{q/(q-1)}
        \Bigg]^{(q-1)/q},
$$
where the universal constant depends on $p$, $q$, and $K$ only.
\end{cor}

\begin{proof}
The norm $\|\cdot\|$ on $X$ can be transferred to $\mathbb{R}^d$
by defining $\|z\|:=\|\sum_{i=1}^dz_ie_i\|$ for $z\in\mathbb{R}^d$.
There is therefore no loss of generality in assuming that 
$X=\mathbb{R}^d$ with the above norm, that $\{e_i\}=\{e_i^*\}$ is the 
standard basis, and that $\|x\|_B$ is the $\ell_q$-norm on $\mathbb{R}^d$, 
as we will do in the sequel for notational simplicity. (We emphasize, 
however, that $\|\cdot\|$ is \emph{not} the Euclidean norm, so that the 
present setting does not reduce to the Euclidean setting considered
previously in section \ref{sec:lqellips}).

As $\|x\|_B$ is the $\ell_q$-norm, we can compute
$$
	\frac{\partial\|x\|_B}{\partial x_i} = 
	\frac{|x_i|^{q-1}}{\|x\|_B^{q-1}}
	\mathop{\mathrm{sign}}(x_i).
$$
By Corollary \ref{cor:grad}, we can write
$$
	B_t = \{
	x\in B:
	\| |x|^{q-1}\mathop{\mathrm{sign}}(x) \|^* \le t||x||_B^{q-1}
	\}.
$$
Now note that for any vectors
$x,y\in\mathbb{R}^d$, we have
$$
	\|x-y\|_B^q =
	\langle |x-y|^{q-1}\mathop{\mathrm{sign}}(x-y),x-y\rangle
	\le
	\||x-y|^{q-1}\mathop{\mathrm{sign}}(x-y)\|^*\|x-y\|.
$$
Moreover, as $|x-y|^{q-1} \le 
2^{(q-2)_+}(|x|^{q-1}+|y|^{q-1})$, we have
$$
	\||x-y|^{q-1}\mathop{\mathrm{sign}}(x-y)\|^* \le
	2^{(q-2)_+}K\||x|^{q-1}+|y|^{q-1}\|^* \le
	2^{1+(q-2)_+}K^2t
$$
for all $x,y\in B_t$ using the unconditional property of the dual basis
$\{e_n^*\}$.  Thus
$$
	\|x-y\|_B^q \le
	2^{1+(q-2)_+}K^2 t\|x-y\|
$$
whenever $x,y\in B_t$, and it remains to invoke Theorem \ref{thm:geompr}.
\end{proof}

We have now given two distinct explanations for the behavior of 
$\ell_q$-ellipsoids observed in section \ref{sec:lqellips}.  When $q\ge 
2$, such sets are $q$-convex and the result follows from the general 
principle described by Corollary \ref{cor:qconvex}.  In this setting, the 
result remains valid when $\|\cdot\|$ is an arbitrary norm.  When $q<2$, 
the observed behavior is described by Corollary \ref{cor:lqelluncond}, 
which exploits a more special geometric property of $\ell_q$-balls. 
In this setting, the result also remains valid for a large class of norms 
$\|\cdot\|$, but we require the additional restriction that the 
underlying basis is unconditional.  It appears that these two cases 
possess a genuinely different geometry, which is completely hidden in the 
statement of Proposition \ref{prop:lqellips}.

\subsection*{Acknowledgments}

The author would like to thank the anonymous referees for helpful 
comments that improved the presentation of this paper.


\end{document}